\documentclass{amsart}
\usepackage{setspace}
\usepackage{a4}
\usepackage{amsthm,latexsym}
\usepackage{amsfonts}
\usepackage{graphicx}
\usepackage{textcomp}
\usepackage{cite}
\usepackage{enumerate}
\usepackage[mathscr]{euscript}
\usepackage{mathtools}
\usepackage{amssymb}
\usepackage{hyperref}
\usepackage{amsmath}
\usepackage{tikz}
\newtheorem{theorem}{Theorem}[section]

\newtheorem{conjecture}[theorem]{Conjecture}
\newtheorem{corollary}[theorem] {Corollary}
\newtheorem{definition}[theorem]{Definition}

\newtheorem{question}[theorem]{Question}
\title{This is the title}

\usepackage{fancyhdr}

\begin{document}
\hrule\hrule\hrule\hrule\hrule
\vspace{0.3cm}	
\begin{center}
{\bf{CONTINUOUS NON-ARCHIMEDEAN AND p-ADIC  WELCH BOUNDS}}\\
\vspace{0.3cm}
\hrule\hrule\hrule\hrule\hrule
\vspace{0.3cm}
\textbf{K. MAHESH KRISHNA}\\
Post Doctoral Fellow \\
Statistics and Mathematics Unit\\
Indian Statistical Institute, Bangalore Centre\\
Karnataka 560 059, India\\
Email: kmaheshak@gmail.com\\

Date: \today
\end{center}

\hrule\hrule
\vspace{0.5cm}
\textbf{Abstract}: We prove the continuous non-Archimedean (resp. p-adic) Banach space and Hilbert space  versions of non-Archimedean  (resp. p-adic) Welch bounds proved by M. Krishna. We formulate continuous non-Archimedean  and p-adic functional  Zauner conjectures.



\hrule


\section{Introduction}
Prof. L. Welch proved the following bounds in  1974 which appears in everyday life \cite{WELCH}.
  \begin{theorem}\cite{WELCH}\label{WELCHTHEOREM} (\textbf{Welch Bounds})
	Let $n> d$.	If	$\{\tau_j\}_{j=1}^n$  is any collection of  unit vectors in $\mathbb{C}^d$, then
	\begin{align*}
		\sum_{1\leq j,k \leq n}|\langle \tau_j, \tau_k\rangle |^{2m}=\sum_{j=1}^n\sum_{k=1}^n|\langle \tau_j, \tau_k\rangle |^{2m}\geq \frac{n^2}{{d+m-1\choose m}}, \quad \forall m \in \mathbb{N}.
	\end{align*}
	In particular,
	\begin{align*}
	\sum_{1\leq j,k \leq n}|\langle \tau_j, \tau_k\rangle |^{2}=		\sum_{j=1}^n\sum_{k=1}^n|\langle \tau_j, \tau_k\rangle |^{2}\geq \frac{n^2}{{d}}.
	\end{align*}
	Further, 
	\begin{align*}
		\text{(\textbf{Higher order Welch bounds})}	\quad		\max _{1\leq j,k \leq n, j\neq k}|\langle \tau_j, \tau_k\rangle |^{2m}\geq \frac{1}{n-1}\left[\frac{n}{{d+m-1\choose m}}-1\right], \quad \forall m \in \mathbb{N}.
	\end{align*}
	In particular,
	\begin{align*}
		\text{(\textbf{First order Welch bound})}\quad 	\max _{1\leq j,k \leq n, j\neq k}|\langle \tau_j, \tau_k\rangle |^{2}\geq\frac{n-d}{d(n-1)}.
	\end{align*}
\end{theorem}
  Theorem \ref{WELCHTHEOREM} is a fundamental result in many areas such as     in the study of root-mean-square (RMS) absolute cross relation of unit vectors   \cite{SARWATEMEETING}, frame potential \cite{BENEDETTOFICKUS, CASAZZAFICKUSOTHERS, BODMANNHAASPOTENTIAL}, 
 correlations \cite{SARWATE},  codebooks \cite{DINGFENG}, numerical search algorithms  \cite{XIA, XIACORRECTION}, quantum measurements 
\cite{SCOTTTIGHT}, coding and communications \cite{TROPPDHILLON, STROHMERHEATH}, code division multiple access (CDMA) systems \cite{CHEBIRA1, CHEBIRA2}, wireless systems \cite{YATES}, compressed/compressive sensing \cite{TAN, VIDYASAGAR, FOUCARTRAUHUT, ELDARKUTYNIOK, BAJWACALDERBANKMIXON, TROPP, SCHNASSVANDERGHEYNST, ALLTOP},  `game of Sloanes' \cite{JASPERKINGMIXON}, equiangular tight frames \cite{SUSTIKTROPP}, equiangular lines \cite{MIXONSOLAZZO, COUTINHOGODSILSHIRAZIZHAN, FICKUSJASPERMIXON, IVERSONMIXON2022, LEMMENSSEIDEL, JIANGTIDORYAOZHAOZHAO, GREAVESKOOLENMUNEMASASZOLLOSI, BALLADRAXLERKEEVASHSUDAKOV, BUKH, DECAEN, GLAZYRINYU, BARGYU, JIANGPOLYANSKII, NEUMAIER, GREAVESSYATRIADIYATSYNA, GODSILROY, CALDERBANKCAMERONKANTORSEIDEL, OKUDAYU, YU, ETTAOUI, COHNKUMARMINTON, GREAVESIVERSONJASPERMIXON1, GREAVESIVERSONJASPERMIXON2}, digital fingerprinting \cite{MIXONQUINNKIYAVASHFICKUS}  etc.

 Theorem \ref{WELCHTHEOREM}  has been upgraded/different proofs were  given   in \cite{CHRISTENSENDATTAKIM, DATTAWELCHLMA, WALDRONSH, WALDRON2003, DATTAHOWARD, ROSENFELD, HAIKINZAMIRGAVISH, EHLEROKOUDJOU, STROHMERHEATH}.  In 2021 M. Krishna derived continuous version of  Theorem \ref{WELCHTHEOREM} \cite{MAHESHKRISHNA}. In 2022 M. Krishna obtained Theorem \ref{WELCHTHEOREM} for Hilbert C*-modules \cite{MAHESHKRISHNA2}, Banach spaces \cite{MAHESHKRISHNA3},  non-Archimedean Hilbert spaces  \cite{MAHESHKRISHNA4}, p-adic Hilbert spaces \cite{MAHESHKRISHNA5} non-Archimedean Banach spaces and p-adic Banach spaces \cite{MAHESHKRISHNA6}.
 
In this paper we derive continuous non-Archimedean (resp. p-adic) Banach space  version  of non-Archimedean (resp. p-adic) functional Welch bounds in Theorems \ref{WELCHNON11F} and  (resp. Theorems  \ref{WELCHNON2}). We formulate  continuous non-Archimedean   Zauner conjectures (Conjectures \ref{NZF} and \ref{NZF2}) and continuous p-adic  Zauner conjectures (Conjectures  \ref{NZ} and \ref{NZ2}).

 \section{Continuous non-Archimedean  Welch bounds}
 In this section we derive continuous non-Archimedean Banach space version of results derived  in \cite{MAHESHKRISHNA6}. Let $\mathbb{K}$ be a non-Archimedean (complete) valued field satisfying 
 \begin{align}\label{FU}
 	\left|\sum_{j=1}^{n}\lambda_j^2\right|=\max_{1\leq j \leq n}|\lambda_j|^2, \quad \forall \lambda_j \in \mathbb{K}, 1\leq j \leq n, \forall n \in \mathbb{N}.
 \end{align}
 For examples of such fields, we refer   \cite{PEREZGARCIASCHIKHOF}. Throughout this section,  we assume that our non-Archimedean field $\mathbb{K}$ satisfies Equation (\ref{FU}). Letter $\mathcal{X}$ stands for a  $d$-dimensional non-Archimedean Banach space over $\mathbb{K}$. Identity operator on $\mathcal{X}$ is denoted by $I_\mathcal{X}$. The dual of $\mathcal{X}$ is denoted by $\mathcal{X}^*$. 
 In the entire paper, $G$ denotes a locally compact group and $\mu_G$ denotes a left Haar measure on $G$. We assume throughout the paper that the diagonal  $\Delta \coloneqq \{(g,g):g \in G\}$ is measurable in $G\times G$. All our integrals are weak non-Archimedean Riemann integrals (see \cite{ROGERS}). 
  \begin{theorem}\label{WELCHNON1}
 	\textbf{(First Order  Continuous Non-Archimedean Functional Welch Bound)} 	Let $\mathcal{X}$ be a $d$-dimensional non-Archimedean Banach space over $\mathbb{K}$.	  Let  $\{\tau_g\}_{g\in G}$ be a   collection in $\mathcal{X}$ and $\{f_g\}_{g\in G}$ be a  collection in $\mathcal{X}^*$ satisfying following conditions.
 	\begin{enumerate}[\upshape (i)]
 		\item For every $x\in \mathcal{X}$ and for every $\phi \in \mathcal{X}^*$, the map 
 		\begin{align*}
 			G \ni g \mapsto f_g(x)\phi(\tau_g) \in \mathcal{X}
 		\end{align*}
 	is measurable and integrable. 
 		\item The map 
 		\begin{align*}
 			S_{f,\tau}:\mathcal{X} \ni x \mapsto \int\limits_{G} f_g(x)\tau_g\,d\mu_G(g) \in \mathcal{X}
 		\end{align*}
 	is a well-defined bounded linear operator.
 	\item The operator $S_{f, \tau}$	is diagonalizable.
 	\item 
 	\begin{align*}
 	\int\limits_{G\times G}|f_g(\tau_h)f_h(\tau_g)|\,d(\mu_G\times \mu_G) (g,h)<\infty.
 	\end{align*}
 	\end{enumerate}
 Then 
 	\begin{align*}
 		\max\left \{\left| \int\limits_{\Delta}f_g(\tau_g)^2 \,d(\mu_G\times \mu_G)(g, g)\right|, \sup_{g, h \in G, g \neq h}|f_g(\tau_h)f_h(\tau_g)|\right\}\geq \frac{1}{|d|}	\left|\int\limits_{G}f_g(\tau_g)\,d\mu_G(g) \right|^2.	
 	\end{align*}
 	In particular, if $f_g(\tau_g) =1$ for all $g\in G$, then 
 	\begin{align*}
 		\text{\textbf{(First order  continnuous non-Archimedean functional Welch bound)}} \\
 		 \max\left\{|(\mu_G\times \mu_G)(\Delta)|, \sup_{g, h \in G, g \neq h}|f_g(\tau_h)f_h(\tau_g)| \right\}\geq \frac{|\mu(G)|^2}{|d|}.	
 	\end{align*}
 \end{theorem}
 \begin{proof}
 	We first note that 
 	\begin{align*}
 		&\operatorname{Tra}(S_{f, \tau})=\int\limits_{G}f_g(\tau_g)\,d\mu_G(g) , \\
 		& \operatorname{Tra}(S^2_{f, \tau})=\int\limits_{G}\int\limits_{G}f_g(\tau_h)f_h(\tau_g)\,d\mu_G(g)\,d\mu_G(h).
 	\end{align*}	
 	Let $\lambda_1, \dots, \lambda_d$ be the diagonal entries in the diagonalization of   $S_{f, \tau}$. Then  using the diagonalizability of $	S_{f, \tau}$ and the non-Archimedean  Cauchy-Schwarz inequality (Theorem 2.4.2 \cite{PEREZGARCIASCHIKHOF}), we get 
 		\begin{align*}
 		\left|\int\limits_{G}f_g(\tau_g)\,d\mu_G(g) \right|^2&=|\operatorname{Tra}(S_{f, \tau})|^2=\left|\sum_{k=1}^d
 		\lambda_k\right|^2\leq |d| 	\left|\sum_{k=1}^d\lambda_k^2 \right|=|d||\operatorname{Tra}(S^2_{f, \tau})|\\
 		&=|d|\left|\int\limits_{G}\int\limits_{G}f_g(\tau_h)f_h(\tau_g)\,d\mu_G(g)\,d\mu_G(h)\right|=|d|\left|\int\limits_{G\times G}f_g(\tau_h)f_h(\tau_g)\,d(\mu_G\times \mu_G) (g,h)\right|\\
 		&=|d|\left| \int\limits_{\Delta}f_g(\tau_g)^2\,d(\mu_G\times \mu_G)(g, g)+\int\limits_{(G\times G)\setminus \Delta}f_g(\tau_h)f_h(\tau_g)\,d(\mu_G\times \mu_G)(g, h)\right|\\
 		&\leq  |d| \max\left \{\left| \int\limits_{\Delta}f_g(\tau_g) ^2\,d(\mu_G\times \mu_G)(g, g)\right|,\left| \int\limits_{(G\times G)\setminus \Delta}f_g(\tau_h)f_h(\tau_g)\,d(\mu_G\times \mu_G)(g, h)\right|\right\}\\
 		&\leq |d| \max\left \{\left| \int\limits_{\Delta}f_g(\tau_g) ^2\,d(\mu_G\times \mu_G)(g, g)\right|, \sup_{g, h \in G, g \neq h}|f_g(\tau_h)f_h(\tau_g)| \right\}.
 	\end{align*}
 	Whenever $f_g(\tau_g) =1$ for all $g\in G$, 
 	\begin{align*}
 		|\mu(G)|^2\leq |d|\max\left\{|(\mu_G\times \mu_G)(\Delta)|, \sup_{g, h \in G, g \neq h}|f_g(\tau_h)f_h(\tau_g)| \right\}.
 	\end{align*}
 \end{proof}
\begin{corollary}\label{COR1}
	\textbf{(First Order  Continuous Non-Archimedean  Welch Bound)}	Let $\mathcal{X}$ be a $d$-dimensional non-Archimedean Hilbert space over $\mathbb{K}$.	  Let  $\{\tau_g\}_{g\in G}$ be a   collection in $\mathcal{X}$  satisfying following conditions.
	\begin{enumerate}[\upshape (i)]
		\item For every $x\in \mathcal{X}$ and for every $\phi \in \mathcal{X}^*$, the map 
		\begin{align*}
			G \ni g \mapsto \langle x, \tau_g \rangle \phi(\tau_g) \in \mathcal{X}
		\end{align*}
		is measurable and integrable. 
		\item The map 
		\begin{align*}
			S_{\tau}:\mathcal{X} \ni x \mapsto \int\limits_{G} \langle x, \tau_g \rangle \tau_g\,d\mu_G(g) \in \mathcal{X}
		\end{align*}
		is a well-defined bounded linear operator.
		\item The operator $S_{\tau}$	is diagonalizable.
		\item 
		\begin{align*}
			\int\limits_{G\times G}|\langle \tau_g, \tau_h \rangle|^2\,d(\mu_G\times \mu_G) (g,h)<\infty.	
		\end{align*}
	\end{enumerate}
	Then 
	\begin{align*}
		\max\left \{\left| \int\limits_{\Delta}\langle \tau_g, \tau_g \rangle ^2 \,d(\mu_G\times \mu_G)(g, g)\right|, \sup_{g, h \in G, g \neq h}|\langle \tau_g, \tau_h \rangle |^2\right\}\geq \frac{1}{|d|}	\left|\int\limits_{G}\langle \tau_g, \tau_g \rangle \,d\mu_G(g) \right|^2.	
	\end{align*}
	In particular, if $\langle \tau_g, \tau_g \rangle  =1$ for all $g\in G$, then 
	\begin{align*}
		\text{\textbf{(First order  continnuous non-Archimedean  Welch bound)}} \\
		\max\left\{|(\mu_G\times \mu_G)(\Delta)|, \sup_{g, h \in G, g \neq h}|\langle \tau_g, \tau_h \rangle |^2 \right\}\geq \frac{|\mu(G)|^2}{|d|}.	
	\end{align*}
\end{corollary}
 Next we obtain higher order continuous non-Archimedean functional Welch bounds. We need the following vector space result.
 \begin{theorem}\cite{COMON, BOCCI}\label{SYMMETRICTENSORDIMENSION}
 	If $\mathcal{V}$ is a vector space of dimension $d$ and $\text{Sym}^m(\mathcal{V})$ denotes the vector space of symmetric m-tensors, then 
 	\begin{align*}
 		\text{dim}(\text{Sym}^m(\mathcal{V}))={d+m-1 \choose m}, \quad \forall m \in \mathbb{N}.
 	\end{align*}
 \end{theorem}
 
 \begin{theorem}\label{WELCHNON11F}
 (\textbf{Higher Order Continuous Non-Archimedean Functional Welch Bounds})
 Let $\mathcal{X}$ be a $d$-dimensional non-Archimedean Banach space over $\mathbb{K}$.	 Let $m \in \mathbb{N}$. Let  $\{\tau_g\}_{g\in G}$ be a   collection in $\mathcal{X}$ and $\{f_g\}_{g\in G}$ be a  collection in $\mathcal{X}^*$ satisfying following conditions.
 \begin{enumerate}[\upshape (i)]
 	\item For every $x\in \text{Sym}^m(\mathcal{X})$ and for every $\phi \in \text{Sym}^m(\mathcal{X})^*$, the map 
 	\begin{align*}
 		G \ni g \mapsto f^{\otimes m}_g(x)\phi(\tau_g^{\otimes m}) \in \text{Sym}^m(\mathcal{X})
 	\end{align*}
 	is measurable and integrable. 
 	\item The map 
 	\begin{align*}
 		S_{f,\tau}:\text{Sym}^m(\mathcal{X}) \ni x \mapsto \int\limits_{G} f^{\otimes m}_g(x)\tau_g^{\otimes m}\,d\mu_G(g) \in \text{Sym}^m(\mathcal{X})
 	\end{align*}
 	is a well-defined bounded linear operator.
 	\item The operator $S_{f, \tau}$	is diagonalizable. 
 		\item 
 	\begin{align*}
 		\int\limits_{G\times G}|f_g(\tau_h)f_h(\tau_g)|^m\,d(\mu_G\times \mu_G) (g,h)<\infty.
 	\end{align*}
 \end{enumerate}
 Then 	
 	\begin{align*}
 		\max\left \{\left| \int\limits_{\Delta}f_g(\tau_g)^{2m}\,d(\mu_G\times \mu_G)(g, g) \right|, \sup_{g, h \in G, g \neq h} |f_g(\tau_h)f_h(\tau_g)|^{m}\right\}\geq \frac{1}{\left|{d+m-1 \choose m}\right|}\left|\int\limits_{G}f_g(\tau_g)^m\,d\mu_G(g) \right|^2.	
 	\end{align*}
 	In particular, if $f_g(\tau_g) =1$ for all $g\in G$, then
 	\begin{align*}
 		\text{\textbf{(Higher order continuous  non-Archimedean functional Welch bounds)}} \\
 		\quad  \max\left\{|(\mu_G\times \mu_G)(\Delta)|, \sup_{g, h \in G, g \neq h}|f_g(\tau_h)f_h(\tau_g)|^{m} \right\}\geq \frac{|\mu(G)|^2}{\left|{d+m-1 \choose m}\right| }.	
 	\end{align*}
 \end{theorem}
 \begin{proof}
 	Let $\lambda_1, \dots, \lambda_{\text{dim}(\text{Sym}^m(\mathcal{X}))}$ be the diagonal entries in the diagonalization of  $S_{f, \tau}$. We note that 
 	 \begin{align*}
 		&b\operatorname{dim(\text{Sym}^m(\mathcal{X}))}=\operatorname{Tra}(bI_{\text{Sym}^m(\mathcal{X})})=\operatorname{Tra}(S_{f, \tau})=\int\limits_{G}f_g^{\otimes m}(\tau_g^{\otimes m})\,d\mu_G(g), \\
 		& b^2\operatorname{dim(\text{Sym}^m(\mathcal{X}))}=\operatorname{Tra}(b^2I_{\text{Sym}^m(\mathcal{X})})=\operatorname{Tra}(S^2_{f, \tau})=\int\limits_{G}\int\limits_{G}f_g^{\otimes m} (\tau^{\otimes m} _h) f^{\otimes m} _h(\tau^{\otimes m} _g)\,d\mu_G(g)\,d\mu_G(h).
 	\end{align*}	
 	 Then 
 		\begin{align*}
 		&\left|\int\limits_{G}f_g(\tau_g)^m\,d\mu_G(g) \right|^2=	\left|\int\limits_{G}f_g^{\otimes m}(\tau_g^{\otimes m})\,d\mu_G(g) \right|^2=|\operatorname{Tra}(S_{f, \tau})|^2=\left|\sum_{k=1}^{\text{dim}(\text{Sym}^m(\mathcal{X}))}
 		\lambda_k\right|^2\\
 		&\leq |\text{dim}(\text{Sym}^m(\mathcal{X}))| 	\left|\sum_{k=1}^{\text{dim}(\text{Sym}^m(\mathcal{X}))}\lambda_k^2 \right|= |\text{dim}(\text{Sym}^m(\mathcal{X}))| |\operatorname{Tra}(S^2_{f, \tau})|\\
 		&=\left|{d+m-1 \choose m}\right||\operatorname{Tra}(S^2_{f,\tau})|=\left|{d+m-1 \choose m}\right|\left|\int\limits_{G}\int\limits_{G}f_g^{\otimes m}(\tau_h^{\otimes m})  f_h^{\otimes m} (\tau_g^{\otimes m})\,d\mu_G(g)\,d\mu_G(h) \right|\\
 		&=\left|{d+m-1 \choose m}\right|\left|\int\limits_{G}\int\limits_{G}f_g(\tau_h) ^mf_h(\tau_g)^m\,d\mu_G(g)\,d\mu_G(h)\right|\\
 		&=\left|{d+m-1 \choose m}\right|\left|\int\limits_{G\times G}f_g(\tau_h) ^mf_h(\tau_g)^m\,d(\mu_G\times \mu_G)(g,h)\right|\\
 		&=\left|{d+m-1 \choose m}\right|\left| \int\limits_{\Delta}f_g(\tau_g)^{2m}\,d(\mu_G\times \mu_G)(g, g)+\int\limits_{(G\times G)\setminus \Delta}f_g(\tau_h) ^mf_h(\tau_g)^m\,d(\mu_G\times \mu_G)(g, h)\right|\\
 		&\leq  |d| \max\left \{\left| \int\limits_{\Delta}f_g(\tau_g) ^{2m}\,d(\mu_G\times \mu_G)(g, g)\right|,\left| \int\limits_{(G\times G)\setminus \Delta}f_g(\tau_h)^mf_h(\tau_g)^m\,d(\mu_G\times \mu_G)(g, h)\right|\right\}\\
 		&\leq \left|{d+m-1 \choose m}\right| \max\left \{\left| \int\limits_{\Delta}f_g(\tau_g)^{2m}\,d(\mu_G\times \mu_G)(g, g) \right|, \sup_{g, h \in G, g \neq h}|f_g(\tau_h) ^mf_h(\tau_g)^m| \right\}\\
 		&=\left|{d+m-1 \choose m}\right| \max\left \{\left| \int\limits_{\Delta}f_g(\tau_g)^{2m}\,d(\mu_G\times \mu_G)(g, g) \right|, \sup_{g, h \in G, g \neq h}|f_g(\tau_h)f_h(\tau_g)|^m \right\}.
 	\end{align*}
 	Whenever $f_g(\tau_g) =1$ for all $g\in G$,
 	\begin{align*}
 		|\mu(G)|^2\leq  \left|{d+m-1 \choose m}\right| \max\left\{|(\mu_G\times \mu_G)(\Delta)|, \sup_{g, h \in G, g \neq h}|f_g(\tau_h)f_h(\tau_g)|^{m} \right\}.
 	\end{align*}
 \end{proof}
\begin{corollary}
(\textbf{Higher Order Continuous Non-Archimedean  Welch Bounds})
Let $\mathcal{X}$ be a $d$-dimensional non-Archimedean Hilbert  space over $\mathbb{K}$.	 Let $m \in \mathbb{N}$. Let  $\{\tau_g\}_{g\in G}$ be a   collection in $\mathcal{X}$  satisfying following conditions.
\begin{enumerate}[\upshape (i)]
	\item For every $x\in \text{Sym}^m(\mathcal{X})$ and for every $\phi \in \text{Sym}^m(\mathcal{X})^*$, the map 
	\begin{align*}
		G \ni g \mapsto \langle x, \tau_g^{\otimes m}\rangle \phi(\tau_g^{\otimes m}) \in \text{Sym}^m(\mathcal{X})
	\end{align*}
	is measurable and integrable. 
	\item The map 
	\begin{align*}
		S_{\tau}:\text{Sym}^m(\mathcal{X}) \ni x \mapsto \int\limits_{G}  \langle x, \tau_g^{\otimes m}\rangle\tau_g^{\otimes m}\,d\mu_G(g) \in \text{Sym}^m(\mathcal{X})
	\end{align*}
	is a well-defined bounded linear operator.
	\item The operator $S_{\tau}$	is diagonalizable.
		\item 
	\begin{align*}
		\int\limits_{G\times G}|\langle \tau_g, \tau_h \rangle|^{2m}\,d(\mu_G\times \mu_G) (g,h)<\infty.	
	\end{align*}
\end{enumerate}
Then 	
\begin{align*}
	\max\left \{\left| \int\limits_{\Delta} \langle \tau_g, \tau_g\rangle^{2m}\,d(\mu_G\times \mu_G)(g, g) \right|, \sup_{g, h \in G, g \neq h} |\langle \tau_g, \tau_h\rangle|^{2m}\right\}\geq \frac{1}{\left|{d+m-1 \choose m}\right|}\left|\int\limits_{G}\langle \tau_g, \tau_g \rangle^m\,d\mu_G(g) \right|^2.	
\end{align*}
In particular, if $\langle \tau_g, \tau_g \rangle  =1$ for all $g\in G$, then
\begin{align*}
	\text{\textbf{(Higher order continuous  non-Archimedean  Welch bounds)}} \\
	\quad  \max\left\{|(\mu_G\times \mu_G)(\Delta)|, \sup_{g, h \in G, g \neq h}|\langle \tau_g, \tau_h \rangle|^{2m} \right\}\geq \frac{|\mu(G)|^2}{\left|{d+m-1 \choose m}\right| }.	
\end{align*}	
\end{corollary}
Motivated from  Theorem \ref{WELCHNON1} and Corollary \ref{COR1} we formulate the following  questions.
 \begin{question}\label{Q1F}
 	\textbf{Let $\mathbb{K}$ be a non-Archimedean field  satisfying Equation (\ref{FU}) and $\mathcal{X}$ be a $d$-dimensional non-Archimedean Banach space over $\mathbb{K}$.
 		For which locally compact  group $G$ with measurable diagonal $\Delta$,   there exist  a collection $\{\tau_g\}_{g\in G}$  in $\mathcal{X}$ and a collection $\{f_g\}_{g\in G}$   in $\mathcal{X}^*$   satisfying the following. 
 			\begin{enumerate}[\upshape(i)]
 				\item For every $x\in \mathcal{X}$ and for every $\phi \in \mathcal{X}^*$, the map 
 				\begin{align*}
 					G \ni g \mapsto f_g(x)\phi(\tau_g) \in \mathcal{X}
 				\end{align*}
 				is measurable and integrable. 
 				\item The map 
 				\begin{align*}
 					S_{f,\tau}:\mathcal{X} \ni x \mapsto \int\limits_{G} f_g(x)\tau_g\,d\mu_G(g) \in \mathcal{X}
 				\end{align*}
 				is a well-defined bounded linear operator.
 				\item The operator $S_{f, \tau}$	is diagonalizable.
 					\item 
 				\begin{align*}
 					\int\limits_{G\times G}|f_g(\tau_h)f_h(\tau_g)|\,d(\mu_G\times \mu_G) (g,h)<\infty.
 				\end{align*}
 			\item $f_g(\tau_g) =1$ for all $g\in G$.
 			\item 
 			\begin{align*}
	\max\left\{|(\mu_G\times \mu_G)(\Delta)|, \sup_{g, h \in G, g \neq h}|f_g(\tau_h)f_h(\tau_g)| \right\}= \frac{|\mu(G)|^2}{|d|}.	
 			\end{align*}
 		\item $\|f_g\| =1$, $\|\tau_g\| =1$ for all $g\in G$.
 	\end{enumerate}}
 \end{question}
\begin{question}\label{CORQ1}
	\textbf{Let $\mathbb{K}$ be a non-Archimedean field  satisfying Equation (\ref{FU}) and $\mathcal{X}$ be a $d$-dimensional non-Archimedean Hilbert space over $\mathbb{K}$.
	For which locally compact  group $G$ with measurable diagonal $\Delta$,   there exists  a collection $\{\tau_g\}_{g\in G}$  in $\mathcal{X}$  satisfying the following. 
	\begin{enumerate}[\upshape (i)]
		\item For every $x\in \mathcal{X}$ and for every $\phi \in \mathcal{X}^*$, the map 
		\begin{align*}
			G \ni g \mapsto \langle x, \tau_g \rangle \phi(\tau_g) \in \mathcal{X}
		\end{align*}
		is measurable and integrable. 
		\item The map 
		\begin{align*}
			S_{\tau}:\mathcal{X} \ni x \mapsto \int\limits_{G} \langle x, \tau_g \rangle \tau_g\,d\mu_G(g) \in \mathcal{X}
		\end{align*}
		is a well-defined bounded linear operator.
		\item The operator $S_{\tau}$	is diagonalizable.
			\item 
		\begin{align*}
			\int\limits_{G\times G}|\langle \tau_g, \tau_h \rangle |^2\,d(\mu_G\times \mu_G) (g,h)<\infty.
		\end{align*}
		\item $\langle \tau_g, \tau_g\rangle =1$ for all $g\in G$.
		\item 
		\begin{align*}
			\max\left\{|(\mu_G\times \mu_G)(\Delta)|, \sup_{g, h \in G, g \neq h}|\langle \tau_g, \tau_h \rangle |^2 \right\}= \frac{|\mu(G)|^2}{|d|}.	
		\end{align*}
\end{enumerate}}	
\end{question}
 A particular case of Questions \ref{Q1F} and \ref{CORQ1}  are  following continuous non-Archimedean  versions of Zauner conjecture (see \cite{APPLEBY123, APPLEBY, ZAUNER, SCOTTGRASSL, FUCHSHOANGSTACEY, RENESBLUMEKOHOUTSCOTTCAVES, APPLEBYSYMM, BENGTSSON, APPLEBYFLAMMIAMCCONNELLYARD, KOPPCON, GOURKALEV, BENGTSSONZYCZKOWSKI, PAWELRUDNICKIZYCZKOWSKI, BENGTSSON123, MAGSINO, MAHESHKRISHNA} for Zauner conjecture in Hilbert spaces,   \cite{MAHESHKRISHNA2}  for Zauner conjecture in Hilbert C*-modules,     \cite{MAHESHKRISHNA3}  for Zauner conjecture in Banach spaces,  \cite{MAHESHKRISHNA4}	 for Zauner conjecture in non-Archimedean Hilbert spaces, \cite{MAHESHKRISHNA5} for Zauner conjecture in p-adic Hilbert spaces and \cite{MAHESHKRISHNA6} for Zauner conjecture for  non-Archimedean Banach spaces and p-adic Banach spaces).
 \begin{conjecture}\label{NZF} \textbf{(Continuous Non-Archimedean Functional Zauner Conjecture)
 		Let $\mathbb{K}$ non-Archimedean field  satisfying Equation (\ref{FU}). Given locally compact group $G$ and for  each $d\in \mathbb{N}$, 	there exist  a collection $\{\tau_g\}_{g\in G}$  in $\mathbb{K}^d$ (w.r.t. any non-Archimedean norm) and a collection $\{f_g\}_{g\in G}$   in $(\mathbb{K}^d)^*$   satisfying the following. 
 		\begin{enumerate}[\upshape(i)]
 			\item For every $x\in \mathbb{K}^d$ and for every $\phi \in (\mathbb{K}^d)^*$, the map 
 			\begin{align*}
 				G \ni g \mapsto f_g(x)\phi(\tau_g) \in \mathbb{K}^d
 			\end{align*}
 			is measurable and integrable. 
 			\item The map 
 			\begin{align*}
 				S_{f,\tau}:\mathbb{K}^d \ni x \mapsto \int\limits_{G} f_g(x)\tau_g\,d\mu_G(g) \in \mathbb{K}^d
 			\end{align*}
 			is a well-defined bounded linear operator.
 			\item The operator $S_{f, \tau}$	is diagonalizable. 
 				\item 
 			\begin{align*}
 				\int\limits_{G\times G}|f_g(\tau_h)f_h(\tau_g)|\,d(\mu_G\times \mu_G) (g,h)<\infty.
 			\end{align*}
 			\item $f_g(\tau_g) =1$ for all $g\in G$.
 			\item 
 			\begin{align*}
 				|(\mu_G\times \mu_G)(\Delta)|= |f_g(\tau_h)f_h(\tau_g)| = \frac{|\mu(G)|^2}{|d|}, \quad \forall g, h \in G, g \neq h.	
 			\end{align*}
 			\item $\|f_g\| =1$, $\|\tau_g\| =1$ for all $g\in G$.
 \end{enumerate}}
 \end{conjecture}
\begin{conjecture}\label{NZF2}
\textbf{(Continuous Non-Archimedean Zauner Conjecture)
	Let $\mathbb{K}$ non-Archimedean field  satisfying Equation (\ref{FU}). Given locally compact group $G$ and for  each $d\in \mathbb{N}$, 	there exists  a collection $\{\tau_g\}_{g\in G}$  in $\mathbb{K}^d$    satisfying the following. 
	\begin{enumerate}[\upshape(i)]
		\item For every $x\in \mathbb{K}^d$ and for every $\phi \in (\mathbb{K}^d)^*$, the map 
		\begin{align*}
			G \ni g \mapsto \langle x, \tau_g \rangle \phi(\tau_g) \in \mathbb{K}^d
		\end{align*}
		is measurable and integrable. 
		\item The map 
		\begin{align*}
			S_{\tau}:\mathbb{K}^d \ni x \mapsto \int\limits_{G} \langle x, \tau_g \rangle\tau_g\,d\mu_G(g) \in \mathbb{K}^d
		\end{align*}
		is a well-defined bounded linear operator.
		\item The operator $S_{\tau}$	is diagonalizable. 
			\item 
		\begin{align*}
			\int\limits_{G\times G}|\langle \tau_g, \tau_h \rangle |^2\,d(\mu_G\times \mu_G) (g,h)<\infty.
		\end{align*}
		\item $\langle \tau_g, \tau_g \rangle=1$ for all $g\in G$.
		\item 
		\begin{align*}
			|(\mu_G\times \mu_G)(\Delta)|= |\langle \tau_g, \tau_h \rangle|^2 = \frac{|\mu(G)|^2}{|d|}, \quad \forall g, h \in G, g \neq h.	
		\end{align*}
\end{enumerate}}	
\end{conjecture}
 There are four bounds which are companions of   Welch bounds in  Hilbert spaces. To recall them we need the notion of Gerzon's bound.
 \begin{definition}\cite{JASPERKINGMIXON}
 	Given $d\in \mathbb{N}$, define \textbf{Gerzon's bound}
 	\begin{align*}
 		\mathcal{Z}(d, \mathbb{K})\coloneqq 
 		\left\{ \begin{array}{cc} 
 			d^2 & \quad \text{if} \quad \mathbb{K} =\mathbb{C}\\
 			\frac{d(d+1)}{2} & \quad \text{if} \quad \mathbb{K} =\mathbb{R}.\\
 		\end{array} \right.
 	\end{align*}	
 \end{definition}
 \begin{theorem}\cite{JASPERKINGMIXON, XIACORRECTION, MUKKAVILLISABHAWALERKIPAAZHANG, SOLTANALIAN, BUKHCOX, CONWAYHARDINSLOANE, HAASHAMMENMIXON, RANKIN}  \label{LEVENSTEINBOUND}
 	Define $\mathbb{K}=\mathbb{R}$ or $\mathbb{C}$ and    $m\coloneqq \operatorname{dim}_{\mathbb{R}}(\mathbb{K})/2$.	If	$\{\tau_j\}_{j=1}^n$  is any collection of  unit vectors in $\mathbb{K}^d$, then
 	\begin{enumerate}[\upshape(i)]
 		\item (\textbf{Bukh-Cox bound})
 		\begin{align*}
 			\max _{1\leq j,k \leq n, j\neq k}|\langle \tau_j, \tau_k\rangle |\geq \frac{\mathcal{Z}(n-d, \mathbb{K})}{n(1+m(n-d-1)\sqrt{m^{-1}+n-d})-\mathcal{Z}(n-d, \mathbb{K})}\quad \text{if} \quad n>d.
 		\end{align*}
 		\item (\textbf{Orthoplex/Rankin bound})	
 		\begin{align*}
 			\max _{1\leq j,k \leq n, j\neq k}|\langle \tau_j, \tau_k\rangle |\geq\frac{1}{\sqrt{d}} \quad \text{if} \quad n>\mathcal{Z}(d, \mathbb{K}).
 		\end{align*}
 		\item (\textbf{Levenstein bound})	
 		\begin{align*}
 			\max _{1\leq j,k \leq n, j\neq k}|\langle \tau_j, \tau_k\rangle |\geq \sqrt{\frac{n(m+1)-d(md+1)}{(n-d)(md+1)}} \quad \text{if} \quad n>\mathcal{Z}(d, \mathbb{K}).
 		\end{align*}
 		\item (\textbf{Exponential bound})
 		\begin{align*}
 			\max _{1\leq j,k \leq n, j\neq k}|\langle \tau_j, \tau_k\rangle |\geq 1-2n^{\frac{-1}{d-1}}.
 		\end{align*}
 	\end{enumerate}	
 \end{theorem}
 Motivated from Theorem \ref{LEVENSTEINBOUND}, Theorem \ref{WELCHNON1} and Corollary \ref{COR1} we ask  the following problems.
 \begin{question}
 	\textbf{Whether there is a continuous   non-Archimedean functional version of Theorem \ref{LEVENSTEINBOUND}? In particular, does there exists a   version of}
 	\begin{enumerate}[\upshape(i)]
 		\item \textbf{continuous non-Archimedean functional Bukh-Cox bound?}
 		\item \textbf{continuous non-Archimedean functional Orthoplex/Rankin bound?}
 		\item \textbf{continuous non-Archimedean functional Levenstein bound?}
 		\item \textbf{continuous non-Archimedean functional Exponential bound?}
 	\end{enumerate}		
 \end{question}
\begin{question}
	\textbf{Whether there is a continuous   non-Archimedean  version of Theorem \ref{LEVENSTEINBOUND}? In particular, does there exists a   version of}
	\begin{enumerate}[\upshape(i)]
		\item \textbf{continuous non-Archimedean  Bukh-Cox bound?}
		\item \textbf{continuous non-Archimedean  Orthoplex/Rankin bound?}
		\item \textbf{continuous non-Archimedean  Levenstein bound?}
		\item \textbf{continuous non-Archimedean  Exponential bound?}
	\end{enumerate}		
\end{question}

\section{Continuous p-adic  Welch bounds}
In this section we derive p-adic Banach space version of results done in \cite{MAHESHKRISHNA5}. Let p be a prime and $\mathbb{Q}_p$ be the field of p-adic numbers. In this section, $\mathcal{X}$ is a $d$-dimensional p-adic Banach space over $\mathbb{Q}_p$. 
\begin{theorem}\label{WELCHNONQ}
\textbf{(First Order  Continuous p-adic Functional Welch Bound)}	Let $p$ be a prime and $\mathcal{X}$ be a $d$-dimensional p-adic Banach space over $\mathbb{Q}_p$.  Let  $\{\tau_g\}_{g\in G}$ be a   collection in $\mathcal{X}$ and $\{f_g\}_{g\in G}$ be a  collection in $\mathcal{X}^*$ satisfying following conditions.
\begin{enumerate}[\upshape (i)]
	\item For every $x\in \mathcal{X}$ and for every $\phi \in \mathcal{X}^*$, the map 
	\begin{align*}
		G \ni g \mapsto f_g(x)\phi(\tau_g) \in \mathcal{X}
	\end{align*}
	is measurable and integrable. 
	\item There exists $b \in \mathbb{Q}_p$ such that 
	\begin{align*}
		 \int\limits_{G} f_g(x)\tau_g\,d\mu_G(g)=bx, \quad \forall x \in \mathcal{X}.
	\end{align*}
	\item 
\begin{align*}
	\int\limits_{G\times G}|f_g(\tau_h)f_h(\tau_g)|\,d(\mu_G\times \mu_G) (g,h)<\infty.
\end{align*}
	\end{enumerate}
Then 
\begin{align*}
 \max\left \{\left|  \int\limits_{\Delta}f_g(\tau_g)^2\,d(\mu_G\times \mu_G)(g, g) \right|, \sup_{g, h \in G, g \neq h}|f_g(\tau_h)f_h(\tau_g)| \right\}\geq \frac{1}{|d|}\left| \int\limits_{G}f_g(\tau_g) \,d\mu_G(g) \right|^2.	
\end{align*}
In particular, if $f_g(\tau_g) =1$ for all $g\in G$, then 
\begin{align*}
 \text{\textbf{(First order  continuous p-adic  functional Welch bound)}} \\
  \max\left\{|(\mu_G\times \mu_G)(\Delta)|, \sup_{g, h \in G, g \neq h}|f_g(\tau_h)f_h(\tau_g)|  \right\}\geq \frac{|\mu(G)|^2}{|d|}.	
\end{align*}
\end{theorem}
\begin{proof}
	Define $S_{f, \tau} : \mathcal{X}\ni x \mapsto \int\limits_{G}f_g(x) \tau_g\,d\mu_G(g) \in \mathcal{X}$. Then 
\begin{align*}
	&bd=\operatorname{Tra}(bI_{\mathcal{X}})=\operatorname{Tra}(S_{f, \tau})= \int\limits_{G}f_g(\tau_g)\,d\mu_G(g) , \\
	& b^2d=\operatorname{Tra}(b^2I_{\mathcal{X}})=\operatorname{Tra}(S^2_{f,\tau})= \int\limits_{G} \int\limits_{G}f_g(\tau_h)f_h(\tau_g)\,d\mu_G(g)\,d\mu_G(h).
\end{align*}	
Therefore 
\begin{align*}
	\left| \int\limits_{G}f_g(\tau_g) \,d\mu_G(g) \right|^2&=|\operatorname{Tra}(S_{f, \tau})|^2=|bd|^2=|d||b^2d|	
	=|d|\left| \int\limits_{G} \int\limits_{G}f_g(\tau_h)f_h(\tau_g)\,d\mu_G(g) \,d\mu_G(h) \right|\\
	&=|d|\left| \int\limits_{G\times G}f_g(\tau_h)f_h(\tau_g)\,d(\mu_G\times \mu_G)(g, h)\right|\\
	&=|d|\left|  \int\limits_{\Delta}f_g(\tau_g)^2\,d(\mu_G\times \mu_G)(g, g)+\int\limits_{(G\times G)\setminus \Delta}f_g(\tau_h)f_h(\tau_g)\,d(\mu_G\times \mu_G)(g, h)\right|\\
&\leq |d| \max\left \{\left|  \int\limits_{\Delta}f_g(\tau_g)^2\,d(\mu_G\times \mu_G)(g, g) \right|, \left|\int\limits_{(G\times G)\setminus \Delta}f_g(\tau_h)f_h(\tau_g)\,d(\mu_G\times \mu_G)(g, h)\right| \right\}\\
&\leq |d| \max\left \{\left|  \int\limits_{\Delta}f_g(\tau_g)^2\,d(\mu_G\times \mu_G)(g, g)\right|, \sup_{g, h \in G, g \neq h}|f_g(\tau_h)f_h(\tau_g)| \right\}.
\end{align*}
Whenever $f_g(\tau_g) =1$ for all $g\in G$,
\begin{align*}
		|\mu(G)|^2\leq |d|\max\left\{|(\mu_G\times \mu_G)(\Delta)|, \sup_{g, h \in G, g \neq h}|f_g(\tau_h)f_h(\tau_g)|  \right\}.
\end{align*}		
\end{proof}
\begin{corollary}\label{CORP1}
\textbf{(First Order  Continuous p-adic  Welch Bound)}	
Let $p$ be a prime and $\mathcal{X}$ be a $d$-dimensional p-adic Hilbert space over $\mathbb{Q}_p$.  Let  $\{\tau_g\}_{g\in G}$ be a   collection in $\mathcal{X}$  satisfying following conditions.
\begin{enumerate}[\upshape (i)]
	\item For every $x\in \mathcal{X}$ and for every $\phi \in \mathcal{X}^*$, the map 
	\begin{align*}
		G \ni g \mapsto \langle x, \tau_g \rangle\phi(\tau_g) \in \mathcal{X}
	\end{align*}
	is measurable and integrable. 
	\item There exists $b \in \mathbb{Q}_p$ such that 
	\begin{align*}
		\int\limits_{G} \langle x, \tau_g \rangle\tau_g\,d\mu_G(g)=bx, \quad \forall x \in \mathcal{X}.
	\end{align*}
	\item 
\begin{align*}
	\int\limits_{G\times G}|\langle \tau_g, \tau_h \rangle |^2\,d(\mu_G\times \mu_G) (g,h)<\infty.
\end{align*}
\end{enumerate}
Then 
\begin{align*}
	\max\left \{\left|  \int\limits_{\Delta}\langle \tau_g, \tau_g \rangle^2\,d(\mu_G\times \mu_G)(g, g) \right|, \sup_{g, h \in G, g \neq h}|\langle \tau_g, \tau_h \rangle|^2 \right\}\geq \frac{1}{|d|}\left| \int\limits_{G}\langle \tau_g, \tau_g \rangle \,d\mu_G(g) \right|^2.	
\end{align*}
In particular, if $\langle \tau_g, \tau_g\rangle =1$ for all $g\in G$, then 
\begin{align*}
	\text{\textbf{(First order  continuous p-adic   Welch bound)}} \\
	\max\left\{|(\mu_G\times \mu_G)(\Delta)|, \sup_{g, h \in G, g \neq h}|\langle \tau_g, \tau_h \rangle|^2  \right\}\geq \frac{|\mu(G)|^2}{|d|}.	
\end{align*}
\end{corollary}
Now we derive higher order version of Theorem \ref{WELCHNONQ}.
  \begin{theorem}\label{WELCHNON2}
(\textbf{Higher Order Continuous p-adic  Functional Welch Bounds})
 Let $p$ be a prime and $\mathcal{X}$ be a $d$-dimensional p-adic Banach space over $\mathbb{Q}_p$. Let $m \in \mathbb{N}$. Let  $\{\tau_g\}_{g\in G}$ be a   collection in $\mathcal{X}$ and $\{f_g\}_{g\in G}$ be a  collection in $\mathcal{X}^*$ satisfying following conditions.
\begin{enumerate}[\upshape (i)]
	\item For every $x\in \text{Sym}^m(\mathcal{X})$ and for every $\phi \in \text{Sym}^m(\mathcal{X})^*$, the map 
	\begin{align*}
		G \ni g \mapsto f^{\otimes m}_g(x)\phi(\tau_g^{\otimes m}) \in \text{Sym}^m(\mathcal{X})
	\end{align*}
	is measurable and integrable. 
	\item There exists $b \in \mathbb{Q}_p$ satisfying 
	\begin{align*}
		\int\limits_{G}f_g^{\otimes m}(x) \tau_g^{\otimes m}\,d\mu_G(g) =bx, \quad \forall x \in \text{Sym}^m(\mathcal{X}).
	\end{align*} 
	\item 
\begin{align*}
	\int\limits_{G\times G}|f_g(\tau_h)f_h(\tau_g)|^m\,d(\mu_G\times \mu_G) (g,h)<\infty.
\end{align*}
\end{enumerate}
Then
\begin{align*}
\max\left \{\left| \int\limits_{\Delta}f_g(\tau_g) ^{2m} \,d(\mu_G\times \mu_G)(g, g)\right|, \sup_{g, h \in G, g \neq h}|f_g(\tau_h)f_h(\tau_g)|^{m}\right\}\geq \frac{1}{\left|{d+m-1 \choose m}\right|}\left|\int\limits_{G}f_g(\tau_g)^m\,d\mu_G(g)  \right|^2.	
\end{align*}
In particular, if $f_g(\tau_g) =1$ for all $g\in G$, then
\begin{align*}
 \text{\textbf{(Higher order p-adic  continuous functional Welch bound)}} \\
   \max \left \{|(\mu_G\times \mu_G)(\Delta)|, \sup_{g, h \in G, g \neq h}|f_g(\tau_h)f_h(\tau_g)|^{m} \right\}\geq \frac{|\mu(G)|^2}{\left|{d+m-1 \choose m}\right| }.	
\end{align*}
 \end{theorem}
  \begin{proof}
 Define $S_{f, \tau} : \text{Sym}^m(\mathcal{X})\ni x \mapsto \int\limits_{G}f_g^{\otimes m}(x)\tau_g^{\otimes m}\,d\mu_G(g) \in \text{Sym}^m(\mathcal{X})$. Then 
    \begin{align*}
    	&b\operatorname{dim(\text{Sym}^m(\mathcal{X}))}=\operatorname{Tra}(bI_{\text{Sym}^m(\mathcal{X})})=\operatorname{Tra}(S_{f, \tau})=\int\limits_{G}f_g^{\otimes m}(\tau_g^{\otimes m})\,d\mu_G(g), \\
    	& b^2\operatorname{dim(\text{Sym}^m(\mathcal{X}))}=\operatorname{Tra}(b^2I_{\text{Sym}^m(\mathcal{X})})=\operatorname{Tra}(S^2_{f, \tau})=\int\limits_{G}\int\limits_{G}f_g^{\otimes m} (\tau^{\otimes m} _h) f^{\otimes m} _h(\tau^{\otimes m} _g)\,d\mu_G(g)\,d\mu_G(h).
    \end{align*}	
Therefore 
\begin{align*}
    	&\left|\int\limits_{G}f_g(\tau_g) ^m \,d\mu_G(g)\right|^2=	\left|\int\limits_{G}f_g^{\otimes m}(\tau_g^{\otimes m})\,d\mu_G(g) \right|^2=|\operatorname{Tra}(S_{f, \tau})|^2=\left|b\operatorname{dim(\text{Sym}^m(\mathcal{X}))}\right|^2\\
    	&=\left|\operatorname{dim(\text{Sym}^m(\mathcal{X}))}\right|\left|b^2\operatorname{dim(\text{Sym}^m(\mathcal{X}))}\right|\\
    	&=\left|\operatorname{dim(\text{Sym}^m(\mathcal{X}))}\right|\left|\int\limits_{G}\int\limits_{G}f_g^{\otimes m}(\tau_h^{\otimes m}) f_h^{\otimes m}( \tau_g^{\otimes m}) \,d\mu_G(g)\,d\mu_G(h)\right|\\
    	&=\left|{d+m-1 \choose m}\right|\left|\int\limits_{G}\int\limits_{G}f_g^{\otimes m}( \tau_h^{\otimes m}) f_h^{\otimes m}(\tau_g^{\otimes m})\,d\mu_G(g)\,d\mu_G(h)\right|\\
    	&=\left|{d+m-1 \choose m}\right|\left|\int\limits_{G}\int\limits_{G}f_g(\tau_h)^m f_h(\tau_g)^m\,d\mu_G(g)\,d\mu_G(h)\right|\\
    	&=\left|{d+m-1 \choose m}\right|\left|\int\limits_{G\times G}f_g(\tau_h)^m f_h(\tau_g)^m\,d(\mu_G \times \mu_G)(g,h)\right|\\
    	&=\left|{d+m-1 \choose m}\right|\left| \int\limits_{\Delta}f_g(\tau_g)^{2m}\,d(\mu_G\times \mu_G)(g, g)+\int \limits_{(G\times G)\setminus \Delta}f_g(\tau_h)^m f_h(\tau_g)^m\,d(\mu_G \times \mu_G)(g,h)\right|\\
    	&\leq \left|{d+m-1 \choose m}\right| \max\left \{\left|  \int\limits_{\Delta}f_g(\tau_g)^{2m}\,d(\mu_G\times \mu_G)(g, g) \right|, \left|\int\limits_{(G\times G)\setminus \Delta}f_g(\tau_h)^mf_h(\tau_g)^m\,d(\mu_G\times \mu_G)(g, h)\right| \right\}\\
    	&\leq \left|{d+m-1 \choose m}\right| \max\left \{\left| \int\limits_{\Delta}f_g(\tau_g)^{2m}\,d(\mu_G\times \mu_G)(g, g) \right|, \sup_{g, h \in G, g \neq h}|f_g(\tau_h)^m f_h(\tau_g)^m| \right\}\\
    	&= \left|{d+m-1 \choose m}\right| \max\left \{\left| \int\limits_{\Delta}f_g(\tau_g)^{2m}\,d(\mu_G\times \mu_G)(g, g) \right|, \sup_{g, h \in G, g \neq h}|f_g(\tau_h)f_h(\tau_g)|^{m}\right\}.
    \end{align*}
    Whenever $f_g(\tau_g) =1$ for all $g\in G$, 
    \begin{align*}
    	|\mu(G)|^2\leq  \left|{d+m-1 \choose m}\right| \max\left\{|(\mu_G\times \mu_G)(\Delta)|, \sup_{g, h \in G, g \neq h}|f_g(\tau_h)f_h(\tau_g)|^{m} \right\}.
    \end{align*}
  \end{proof}
\begin{corollary}
(\textbf{Higher Order Continuous p-adic   Welch Bounds})
Let $p$ be a prime and $\mathcal{X}$ be a $d$-dimensional p-adic Hilbert space over $\mathbb{Q}_p$. Let $m \in \mathbb{N}$. Let  $\{\tau_g\}_{g\in G}$ be a   collection in $\mathcal{X}$  satisfying following conditions.
\begin{enumerate}[\upshape (i)]
	\item For every $x\in \text{Sym}^m(\mathcal{X})$ and for every $\phi \in \text{Sym}^m(\mathcal{X})^*$, the map 
	\begin{align*}
		G \ni g \mapsto \langle x, \tau^{\otimes m}_g\rangle \phi(\tau_g^{\otimes m}) \in \text{Sym}^m(\mathcal{X})
	\end{align*}
	is measurable and integrable. 
	\item There exists $b \in \mathbb{Q}_p$ satisfying 
	\begin{align*}
		\int\limits_{G}\langle x, \tau_g^{\otimes m}\rangle  \tau_g^{\otimes m}\,d\mu_G(g) =bx, \quad \forall x \in \text{Sym}^m(\mathcal{X}).
	\end{align*} 
	\item 
\begin{align*}
	\int\limits_{G\times G}|\langle \tau_g, \tau_h \rangle |^{2m}\,d(\mu_G\times \mu_G) (g,h)<\infty.
\end{align*}
\end{enumerate}
Then
\begin{align*}
	\max\left \{\left| \int\limits_{\Delta}\langle \tau_g, \tau_g\rangle  ^{2m} \,d(\mu_G\times \mu_G)(g, g)\right|, \sup_{g, h \in G, g \neq h}|\langle \tau_g, \tau_h\rangle |^{m}\right\}\geq \frac{1}{\left|{d+m-1 \choose m}\right|}\left|\int\limits_{G}\langle \tau_g, \tau_g\rangle ^m\,d\mu_G(g)  \right|^2.	
\end{align*}
In particular, if $\langle \tau_g, \tau_g\rangle  =1$ for all $g\in G$, then
\begin{align*}
	\text{\textbf{(Higher order continuous p-adic   Welch bound)}} \\
	\max \left \{|(\mu_G\times \mu_G)(\Delta)|, \sup_{g, h \in G, g \neq h}|\langle \tau_g, \tau_h\rangle |^{2m} \right\}\geq \frac{|\mu(G)|^2}{\left|{d+m-1 \choose m}\right| }.	
\end{align*}	
\end{corollary}
 Using Theorem \ref{WELCHNONQ} and Corollary \ref{CORP1} we ask the  following questions.
\begin{question}\label{Q1}
		\textbf{Let $p$ be a prime and $\mathcal{X}$ be a $d$-dimensional p-adic  Banach space over $\mathbb{Q}_p$. For which locally compact  group $G$ with measurable diagonal $\Delta$,   there exist  a collection $\{\tau_g\}_{g\in G}$  in $\mathcal{X}$ and a collection $\{f_g\}_{g\in G}$   in $\mathcal{X}^*$   satisfying the following. 
		\begin{enumerate}[\upshape(i)]
			\item For every $x\in \mathcal{X}$ and for every $\phi \in \mathcal{X}^*$, the map 
			\begin{align*}
				G \ni g \mapsto f_g(x)\phi(\tau_g) \in \mathcal{X}
			\end{align*}
			is measurable and integrable. 
			\item There exists $b \in \mathbb{Q}_p$ such that 
			\begin{align*}
				\int\limits_{G} f_g(x)\tau_g\,d\mu_G(g)=bx, \quad \forall x \in \mathcal{X}.
			\end{align*}
			\item 
		\begin{align*}
			\int\limits_{G\times G}|f_g(\tau_h)f_h(\tau_g)|\,d(\mu_G\times \mu_G) (g,h)<\infty.
		\end{align*}
			\item $f_g(\tau_g) =1$ for all $g\in G$.
			\item 
			\begin{align*}
				\max\left\{|(\mu_G\times \mu_G)(\Delta)|, \sup_{g, h \in G, g \neq h}|f_g(\tau_h)f_h(\tau_g)| \right\}= \frac{|\mu(G)|^2}{|d|}.	
			\end{align*}
			\item $\|f_g\| =1$, $\|\tau_g\| =1$ for all $g\in G$.
	\end{enumerate}}
\end{question}
\begin{question}\label{Q2}
	\textbf{Let $p$ be a prime and $\mathcal{X}$ be a $d$-dimensional p-adic  Hilbert space over $\mathbb{Q}_p$. For which locally compact  group $G$ with measurable diagonal $\Delta$,   there exists  a collection $\{\tau_g\}_{g\in G}$  in $\mathcal{X}$    satisfying the following. 
	\begin{enumerate}[\upshape (i)]
		\item For every $x\in \mathcal{X}$ and for every $\phi \in \mathcal{X}^*$, the map 
		\begin{align*}
			G \ni g \mapsto \langle x, \tau_g \rangle \phi(\tau_g) \in \mathcal{X}
		\end{align*}
		is measurable and integrable. 
		\item There exists $b \in \mathbb{Q}_p$ such that 
		\begin{align*}
			\int\limits_{G} \langle x, \tau_g \rangle\tau_g\,d\mu_G(g)=bx, \quad \forall x \in \mathcal{X}.
		\end{align*}
		\item 
	\begin{align*}
		\int\limits_{G\times G}|\langle \tau_g, \tau_h \rangle |^2\,d(\mu_G\times \mu_G) (g,h)<\infty.
	\end{align*}
		\item $\langle \tau_g, \tau_g\rangle =1$ for all $g\in G$.
		\item 
		\begin{align*}
			\max\left\{|(\mu_G\times \mu_G)(\Delta)|, \sup_{g, h \in G, g \neq h}|\langle \tau_g, \tau_h \rangle |^2 \right\}= \frac{|\mu(G)|^2}{|d|}.	
		\end{align*}
		\item  $\|\tau_g\| =1$ for all $g\in G$.
\end{enumerate}}		
\end{question}
A particular case of Questions \ref{Q1} and \ref{Q2} are  following continuous  p-adic  Zauner conjectures.
\begin{conjecture}\label{NZ} \textbf{(Continuous p-adic Functional Zauner Conjecture)
	Let $p$ be a prime. Given locally compact group $G$ and for  each $d\in \mathbb{N}$, 	there exist  a collection $\{\tau_g\}_{g\in G}$  in $\mathbb{Q}_p^d$ (w.r.t. any non-Archimedean norm) and a collection $\{f_g\}_{g\in G}$   in $(\mathbb{Q}_p^d)^*$   satisfying the following. 
\begin{enumerate}[\upshape(i)]
	\item For every $x\in \mathbb{Q}_p^d$ and for every $\phi \in (\mathbb{Q}^d_p)^*$, the map 
	\begin{align*}
		G \ni g \mapsto f_g(x)\phi(\tau_g) \in \mathbb{Q}_p^d
	\end{align*}
	is measurable and integrable. 
	\item There exists $b \in \mathbb{Q}_p$ such that 
	\begin{align*}
		\int\limits_{G} f_g(x)\tau_g\,d\mu_G(g)=bx, \quad \forall x \in \mathbb{Q}_p^d.
	\end{align*}
	\item 
\begin{align*}
	\int\limits_{G\times G}|f_g(\tau_h)f_h(\tau_g)|\,d(\mu_G\times \mu_G) (g,h)<\infty.
\end{align*}
	\item $f_g(\tau_g) =1$ for all $g\in G$.
	\item 
	\begin{align*}
		|(\mu_G\times \mu_G)(\Delta)|= |f_g(\tau_h)f_h(\tau_g)| = \frac{|\mu(G)|^2}{|d|}, \quad \forall g, h \in G, g \neq h.	
	\end{align*}
	\item $\|f_g\| =1$, $\|\tau_g\| =1$ for all $g\in G$.
\end{enumerate}}	
\end{conjecture}
\begin{conjecture}\label{NZ2}
\textbf{(Continuous p-adic  Zauner Conjecture)
	Let $p$ be a prime. Given locally compact group $G$ and for  each $d\in \mathbb{N}$, 	there exists  a collection $\{\tau_g\}_{g\in G}$  in $\mathbb{Q}_p^d$    satisfying the following. 
	\begin{enumerate}[\upshape(i)]
		\item For every $x\in \mathbb{Q}_p^d$ and for every $\phi \in (\mathbb{Q}^d_p)^*$, the map 
		\begin{align*}
			G \ni g \mapsto \langle x , \tau_g \rangle\phi(\tau_g) \in \mathbb{Q}_p^d
		\end{align*}
		is measurable and integrable. 
		\item There exists $b \in \mathbb{Q}_p$ such that 
		\begin{align*}
			\int\limits_{G} \langle x, \tau_g \rangle\tau_g\,d\mu_G(g)=bx, \quad \forall x \in \mathbb{Q}_p^d.
		\end{align*}
		\item 
	\begin{align*}
		\int\limits_{G\times G}|\langle \tau_g, \tau_h \rangle |^2\,d(\mu_G\times \mu_G) (g,h)<\infty.
	\end{align*}
		\item $\langle \tau_g , \tau_g \rangle =1$ for all $g\in G$.
		\item 
		\begin{align*}
			|(\mu_G\times \mu_G)(\Delta)|= |\langle \tau_g , \tau_h \rangle|^2 = \frac{|\mu(G)|^2}{|d|}, \quad \forall g, h \in G, g \neq h.	
		\end{align*}
		\item $\|\tau_g\| =1$ for all $g\in G$.
\end{enumerate}}		
\end{conjecture}
Theorem \ref{LEVENSTEINBOUND}, Theorem \ref{WELCHNONQ}  and Corollary \ref{CORP1} give  the following problems.
\begin{question}
	\textbf{Whether there is a   continuous p-adic functional version of Theorem \ref{LEVENSTEINBOUND}? In particular, does there exists a   version of}
	\begin{enumerate}[\upshape(i)]
		\item \textbf{continuous p-adic functional Bukh-Cox bound?}
		\item \textbf{continuous p-adic functional Orthoplex/Rankin bound?}
		\item \textbf{continuous p-adic functional Levenstein bound?}
		\item \textbf{continuous p-adic functional Exponential bound?}
	\end{enumerate}		
\end{question}
\begin{question}
	\textbf{Whether there is a   continuous p-adic  version of Theorem \ref{LEVENSTEINBOUND}? In particular, does there exists a   version of}
	\begin{enumerate}[\upshape(i)]
		\item \textbf{continuous p-adic  Bukh-Cox bound?}
		\item \textbf{continuous p-adic Orthoplex/Rankin bound?}
		\item \textbf{continuous p-adic  Levenstein bound?}
		\item \textbf{continuous p-adic  Exponential bound?}
	\end{enumerate}		
\end{question}

 \bibliographystyle{plain}
 \bibliography{reference.bib}

\end{document}